\documentclass[graybox]{svmult}

\usepackage{type1cm}        
\usepackage[bottom]{footmisc}

\usepackage{amsmath,amsfonts,amssymb}
\usepackage{graphicx} 

\newcommand{\bu}{\mathbf{u}}
\newcommand{\bff}{\mathbf{f}}
\newcommand{\bp}{\mathbf{p}}
\newcommand{\bv}{\mathbf{v}}

\newcommand{\wlam}{\widetilde{\lambda}}
\newcommand{\wm}{\widetilde{m}}
\newcommand{\wt}{\widetilde{t}}
\newcommand{\wT}{\widetilde{T}}
\newcommand{\weps}{\widetilde{\varepsilon}}
\newcommand{\wla}{\widetilde{\lambda}}
\newcommand{\bareps}{\bar{\varepsilon}}
\newcommand{\Span}{\mathrm{span}}
\newcommand{\Diag}{\mathrm{diag}}

\makeindex             

\begin{document}
	
	\title*{On the asymptotic optimality of spectral coarse spaces}
	
	\author{Gabriele Ciaramella and Tommaso Vanzan}
	\authorrunning{G. Ciaramella and T. Vanzan}
	\institute{Gabriele Ciaramella \at Politecnico di Milano \email{gabriele.ciaramella@polimi.it}
		\and Tommaso Vanzan \at CSQI Chair, EPFL Lausanne \email{tommaso.vanzan@epfl.ch}}
	%
	%
	\maketitle
	
	\abstract*{The goal of this paper is concerned with the asymptotic optimality of
spectral coarse spaces for two-level iterative methods. Spectral coarse spaces, namely
coarse spaces obtained as the span of the slowest modes of the used one-level smoother,
are known to be very efficient and, in some cases, optimal. 
However, the results of this paper show that spectral coarse spaces do not
necessarily minimize the asymptotic contraction factor of a two-level iterative method.
Moreover, results of numerical experiments show that there exist coarse spaces that are
asymptotically more efficient and lead to preconditioned systems with improved conditioning
properties.}

\section{Introduction}\label{Ciaramella_mini_10_sec:intro}
\vspace*{-4mm}

The goal of this work is to study the asymptotic optimality of spectral
coarse spaces for two-level iterative methods. 
In particular, we consider a linear system $A \bu = \bff$, 
where $A \in \mathbb{R}^{n \times n}$ and $\bff \in \mathbb{R}^n$,
and a two-level method that, given an iterate $\bu^k$, computes the new vector $\bu^{k+1}$ as
\begin{align}
\bu^{k+1/2}&=G\bu^{k}+M^{-1}\bff, && \text{(smoothing step)} \label{Ciaramella_mini_10_eq:SM} \\
\bu^{k+1}&=\bu^{k+1/2}+ PA_c^{-1}R(\bff - A\bu^{k+1/2}). &&  \text{(coarse correction)}  \label{Ciaramella_mini_10_eq:CC}
\end{align}
The smoothing step \eqref{Ciaramella_mini_10_eq:SM} 
is based on the splitting $A=M-N$, where $M$ is the preconditioner,
and $G=M^{-1}N$ the iteration matrix.
The correction step \eqref{Ciaramella_mini_10_eq:CC} is characterized by 
prolongation and restriction matrices $P \in \mathbb{R}^{n \times m}$ and $R=P^\top$,
and a coarse matrix $A_c = RAP$. 
The columns
of $P$ are linearly independent vectors spanning the coarse space
$V_c := \Span \, \{ \bp_1 , \dots, \bp_m \}$.
The convergence of the one-level iteration \eqref{Ciaramella_mini_10_eq:SM}
is characterized by the eigenvalues of $G$,
$\lambda_j$, $j=1,\dots,n$ (sorted in descending order by magnitude).
The convergence of the two-level iteration 
\eqref{Ciaramella_mini_10_eq:SM}-\eqref{Ciaramella_mini_10_eq:CC}
depends on the spectrum of the iteration matrix $T$,
obtained by replacing \eqref{Ciaramella_mini_10_eq:SM} into
\eqref{Ciaramella_mini_10_eq:CC} and rearranging terms:

\vspace*{-2mm}
\begin{equation}\label{Ciaramella_mini_10_eq:2L}
T = [ I - P ( RAP)^{-1} R A ] G.
\end{equation}
\vspace*{-4mm}

\noindent
The goal of this short paper is to answer, though partially, to the fundamental question: 
{\bf given an integer $m$, what is the coarse space of dimension $m$ which minimizes the spectral radius $\rho(T)$?}
Since step \eqref{Ciaramella_mini_10_eq:CC} aims at correcting
the error components that the smoothing step \eqref{Ciaramella_mini_10_eq:SM} is 
not able to reduce (or eliminate), it is intuitive to think that an optimal
coarse space $V_c$ is obtained by defining $\bp_j$ as the eigenvectors of $G$ corresponding to the
$m$ largest (in modulus) eigenvalues. We call such a $V_c$ \textit{spectral coarse space}.
Following the idea of correcting the `badly converging' modes of $G$, several papers 
proposed new, and in some sense optimal, coarse spaces.
In the context of domain decomposition methods, we refer, e.g., to 
\cite{gander2014new,GHS2018,gander2019song}, where efficient coarse spaces have been designed for
parallel, restricted additive and additive Schwarz methods. 
In the context of multigrid methods, it is worth to mention the work
\cite{katrutsa2017deep}, where the interpolation weights
are optimized using an approach based on deep-neural networks.
Fundamental results are presented in \cite{xu_zikatanov_2017}: for a symmetric $A$,
it is proved that the coarse space of size $m$ that minimizes the energy norm of $T$,
namely $\| T \|_A$, is the span of the $m$ eigenvectors of $\overline{M}A$ corresponding to the $m$
lowest eigenvalues. Here, $\overline{M} := M^{-1} + M^{-\top} - M^{-\top}AM^{-1}$ is symmetric 
and assumed positive definite.
If $M$ is symmetric, a direct calculation gives
$\overline{M}A=2M^{-1}A-(M^{-1}A)^2$. Using that $M^{-1}A=I-G$, one can show 
that the $m$ eigenvectors associated to the lowest $m$ eigenvalues of $\overline{M}A$ correspond to the $m$ slowest modes of $G$.
Hence, the optimal coarse space proposed in \cite{xu_zikatanov_2017}
is a spectral coarse space.
The sharp result of \cite{xu_zikatanov_2017} provides a concrete optimal choice
of $V_c$ minimizing $\| T\|_A$. This is generally an upper bound for
the asymptotic convergence factor $\rho(T)$.
As we will see in Section \ref{Ciaramella_mini_10_sec:pert}, 
choosing the spectral coarse space, one gets $\rho(T)=|\lambda_{m+1}|$.
The goal of this work is to show that this is not necessarily the optimal asymptotic convergence factor.
In Section \ref{Ciaramella_mini_10_sec:pert}, we perform a detailed optimality analysis
for the case $m=1$. The asymptotic optimality of coarse spaces for $m\geq 1$ is studied numerically
in Section \ref{Ciaramella_mini_10_sec:opt}. Interestingly, we will see that by optimizing 
$\rho(T)$ one constructs coarse spaces that lead to preconditioned matrices with better condition numbers.

\vspace*{-2mm}
\section{A perturbation approach}\label{Ciaramella_mini_10_sec:pert}
\vspace*{-4mm}

Let $G$ be diagonalizable with eigenpairs $(\lambda_j,\bv_j)$, $j=1,\dots,n$.
Suppose that $\bv_j$ are also eigenvectors of $A$:
$A \bv_j = \wlam_j \bv_j$. Concrete examples where these hypotheses are fulfilled
are given in Section \ref{Ciaramella_mini_10_sec:opt}.
Assume that $\textrm{rank} \, P = m$ ($\textrm{dim}\, V_c=m$).
For any eigenvector $\bv_j$, 
we can write the vector $T \bv_j$ as

\vspace*{-3mm}
\begin{equation}\label{Ciaramella_mini_10_eq:wT}
T \bv_j = \sum_{\ell=1}^n \wt_{j,\ell} \bv_\ell, \: j=1,\dots,n.
\end{equation}
\vspace*{-2mm}

\noindent
If we denote by $\wT \in \mathbb{R}^{n \times n}$ the matrix of entries $\wt_{j,\ell}$, 
and define $V:=[\bv_1,\dots,\bv_n]$, then \eqref{Ciaramella_mini_10_eq:wT} becomes $TV=V\wT^\top$. 
Since $G$ is diagonalizable, $V$ is invertible, and thus
$T$ and $\widetilde{T}^\top$ are similar. Hence, $T$ and $\widetilde{T}$ have the same spectrum.
We can now prove the following lemma.

\vspace*{-1mm}
\begin{lemma}[Characterization of $\wT$]\label{Ciaramella_mini_10_lemma:2}
Given an index $\wm \geq m$ and assume that $V_c := \Span \, \{ \bp_1 , \dots, \bp_m \}$ satisfies
\begin{equation}\label{Ciaramella_mini_10_eq:ass1}
V_c \subseteq \Span\, \{ \bv_j \}_{j=1}^{\wm}
\text{ and } V_c \cap \{ \bv_j \}_{j=\wm+1}^n = \{ 0 \}.
\end{equation}
Then, it holds that
\begin{equation}\label{Ciaramella_mini_10_eq:wTT}
\begin{aligned}[c]
\begin{bmatrix}
\wT_{\wm} & 0 \\
X & \Lambda_{\wm}\\
\end{bmatrix},
\end{aligned}
\qquad
\begin{aligned}[c]
&\Lambda_{\wm} = \Diag\, (\lambda_{\wm+1},\dots,\lambda_{n}), \\
& \wT_{\wm} \in \mathbb{R}^{\wm \times \wm}, X \in \mathbb{R}^{(n-\wm) \times \wm}.
\end{aligned}
\end{equation}
\end{lemma}

\begin{proof}
The hypothesis \eqref{Ciaramella_mini_10_eq:ass1}
guarantees that $\Span\, \{ \bv_j \}_{j=1}^{\wm}$ is invariant under the action of $T$. 
Hence, $T \bv_j \in \Span\, \{ \bv_j \}_{j=1}^{\wm}$ for $j=1,\dots,\wm$,
and, using \eqref{Ciaramella_mini_10_eq:wT}, one gets that
$\wt_{j,\ell}=0$ for $j=1,\dots,\wm$ and $\ell=\wm+1,\dots,n$.
Now, consider any $j>\wm$. A direct calculation using \eqref{Ciaramella_mini_10_eq:wT} reveals that
$T\bv_j = G\bv_j - P (RAP)^{-1}RAG \bv_j = \lambda_j \bv_j - \sum_{\ell=1}^{\wm} x_{j-\wm,\ell} \bv_{\ell}$, where $x_{i,k}$ are the elements of $X\in \mathbb{R}^{(n-\wm)\times\wm }$.
Hence, the structure \eqref{Ciaramella_mini_10_eq:wTT} follows.
\end{proof}

Notice that, if \eqref{Ciaramella_mini_10_eq:ass1} holds,
then Lemma \ref{Ciaramella_mini_10_lemma:2}
allows us to study the properties of $T$ using the matrix $\wT$ and
its structure \eqref{Ciaramella_mini_10_eq:wTT}, and hence $\wT_{\wm}$.

Let us now turn to the questions posed in Section \ref{Ciaramella_mini_10_sec:intro}.
Assume that $\bp_j=\bv_j$, $j=1,\dots,m$, namely $V_c = \Span \, \{ \bv_j\}_{j=1}^m$.
In this case, \eqref{Ciaramella_mini_10_eq:ass1} holds with $\wm = m$, and 
a simple argument\footnote{
Let ${\bf v}_j$ be an eigenvector of $A$
with $j \in \{1,\dots,m\}$. Denote
by ${\bf e}_j \in \mathbb{R}^n$ the $j$th canonical vector.
Since $P {\bf e}_j={\bf v}_j$,
$RAP {\bf e}_j = RA {\bf v}_j$. This is equivalent to 
${\bf e}_j = (RAP)^{-1} RA {\bf v}_j$, which gives
$T\bv_j = \lambda_j(\bv_j - P (RAP)^{-1}RA \bv_j) 
= \lambda_j( \bv_j - P {\bf e}_j) =0$.
}
leads to
$\wT_{\wm}=0$, 
$\wT = 
\begin{bmatrix}
0 & 0 \\
X & \Lambda_{\wm}\\
\end{bmatrix}$. The spectrum of $\wT$ is $\{0,\lambda_{m+1},\dots,\lambda_n\}$.
This means that $V_c \subset \textrm{kern} \, T$ and $\rho(T)=|\lambda_{m+1}|$.
Let us now perturb the coarse space $V_c$ using the eigenvector $\bv_{m+1}$, that is
$V_c(\varepsilon) := \Span \, \{ \bv_j + \varepsilon \, \bv_{m+1} \}_{j=1}^m$.
Clearly, $\text{dim}\, V_c(\varepsilon) = m$ for any $\varepsilon \in \mathbb{R}$.
In this case, \eqref{Ciaramella_mini_10_eq:ass1} holds with $\wm = m+1$ and 
$\wT$ becomes 
\begin{equation}\label{Ciaramella_mini_10_thm:TTTT}
\wT(\varepsilon) = 
\begin{bmatrix}
\wT_{\wm}(\varepsilon) & 0 \\
X(\varepsilon) & \Lambda_{\wm}\\
\end{bmatrix},
\end{equation}
where we make explicit the dependence on $\varepsilon$.
Notice that $\varepsilon=0$ clearly leads to 
$\wT_{\wm}(0)=\text{ diag}\, (0,\dots,0,\lambda_{m+1}) \in \mathbb{R}^{\wm \times \wm}$,
and we are back to the unperturbed case with $\wT(0)=\wT$ having spectrum $\{0,\lambda_{m+1},\dots,\lambda_n\}$. Now, notice that
$\min_{\varepsilon \in \mathbb{R}} \rho(\wT(\varepsilon)) \leq \rho(\wT(0)) = | \lambda_{m+1} |$.
Thus, it is natural to ask the question: is this inequality strict?
Can one find an $\weps \neq 0$ such that
$\rho(\wT(\weps))=\min_{\varepsilon \in \mathbb{R}} \rho(\wT(\varepsilon))<\rho(\wT(0))$
holds? If the answer is positive, then we can conclude that choosing the coarse vectors
equal to the dominating eigenvectors of $G$ is not an optimal choice.
The next key result shows that, in the case $m=1$, the answer is positive.

\begin{theorem}[Perturbation of $V_c$]\label{Ciaramella_mini_10_thm:perturb}
Let $(\bv_1,\lambda_1)$, $(\bv_2,\lambda_2)$ and $(\bv_3,\lambda_3)$ be three real eigenpairs of $G$, 
$G \bv_j = \lambda_j \bv_j$ such that
with $0<|\lambda_3|<|\lambda_2| \leq |\lambda_1|$ and $\| \bv_j \|_2 =1$, $j=1,2$. 
Denote by $\wla_j \in \mathbb{R}$ the eigenvalues of $A$ corresponding to $\bv_j$, and assume that
$\wla_1\wla_2>0$. Define
$V_c := \Span\,\{ \bv_1 + \varepsilon \bv_2 \}$ with $\varepsilon \in \mathbb{R}$, and
$\gamma := \bv_1^\top \bv_2 \in [-1,1]$. Then
\begin{itemize}\itemsep0em
\item[{\rm (A)}]$\,$ The spectral radius of $\wT(\varepsilon)$ is
$\rho(\wT(\varepsilon))=\max\{ |\lambda(\varepsilon,\gamma)| , | \lambda_3 | \}$, where
\begin{equation}\label{Ciaramella_mini_10_thm:lam}
\lambda(\varepsilon,\gamma) = \frac{\lambda_1 \wla_2 \varepsilon^2 + \gamma(\lambda_1 \wla_2 + \lambda_2 \wla_1)\varepsilon + \lambda_2 \wla_1}{\wla_2 \varepsilon^2 + \gamma (\wla_1+\wla_2)\varepsilon + \wla_1}.
\end{equation}

\item[{\rm (B)}]$\,$ Let $\gamma=0$.
If $\lambda_1>\lambda_2>0$ or $0>\lambda_2>\lambda_1$, then
$\min\limits _{\varepsilon \in \mathbb{R}} \rho(\wT(\varepsilon)) = \rho(\wT(0))$.

\item[{\rm (C)}]$\,$ Let $\gamma=0$,
If $\lambda_2>0>\lambda_1$ or $\lambda_1>0>\lambda_2$, then there exists an $\weps \neq 0$ such that
$\rho(\wT(\weps)) = |\lambda_3| = \min\limits_{\varepsilon \in \mathbb{R}} \rho(\wT(\varepsilon)) < \rho(\wT(0))$.

\item[{\rm (D)}]$\,$ Let $\gamma\neq 0$.
If $\lambda_1>\lambda_2>0$ or $0>\lambda_2>\lambda_1$, then
there exists an $\weps \neq 0$ such that $|\lambda(\weps,\gamma)|<|\lambda_2|$
and hence
$\rho(\wT(\weps)) = \max\{|\lambda(\weps,\gamma)|,|\lambda_3|\} < \rho(\wT(0))$.

\item[{\rm (E)}]$\,$ Let $\gamma\neq 0$.
If $\lambda_2>0>\lambda_1$ or $\lambda_1>0>\lambda_2$, then 
there exists an $\weps \neq 0$ such that
$\rho(\wT(\weps)) = |\lambda_3| = \min\limits _{\varepsilon \in \mathbb{R}} \rho(\wT(\varepsilon)) < \rho(\wT(0))$.

\end{itemize}
\end{theorem}

\begin{proof}
Since $m=1$, a direct calculation allows us to compute the matrix
$$\wT_{\wm}(\varepsilon)=\begin{bmatrix}
 \lambda_1 - \frac{\lambda_1\wla_1(1+\varepsilon \gamma)}{g} &  -\varepsilon \frac{\lambda_1\wla_1(1+\varepsilon \gamma)}{g} \\
- \frac{\lambda_2\wla_2(\varepsilon + \gamma)}{g} & \lambda_2 - \frac{(\varepsilon\lambda_2\wla_2)(\varepsilon + \gamma)}{g} \\
\end{bmatrix},$$ 
where $g=\wla_1 + \varepsilon \gamma[ \wla_1+\wla_2] + \varepsilon^2 \wla_2$.
The spectrum of this matrix is $\{0, \lambda(\varepsilon,\gamma)\}$,
with $\lambda(\varepsilon,\gamma)$ given in \eqref{Ciaramella_mini_10_thm:lam}.
Hence, point ${\rm (A)}$ follows recalling \eqref{Ciaramella_mini_10_thm:TTTT}.

To prove points ${\rm (B)}$, ${\rm (C)}$, ${\rm (D)}$ and ${\rm (E)}$ we use some properties
of the map $\varepsilon \mapsto \lambda(\varepsilon,\gamma)$. First, we notice that 
\begin{equation}\label{Ciaramella_mini_10_thm:prop}
\lambda(0,\gamma)=\lambda_2, \; \lim_{\varepsilon \rightarrow \pm \infty} \lambda(\varepsilon,\gamma) = \lambda_1,
\; \lambda(\varepsilon,\gamma)=\lambda(-\varepsilon,-\gamma).
\end{equation}
Second, the derivative of $\lambda(\varepsilon,\gamma)$ with respect to $\varepsilon$ is
\begin{equation}\label{Ciaramella_mini_10_thm:der}
\frac{d \lambda(\varepsilon,\gamma)}{d \varepsilon}
= \frac{(\lambda_1-\lambda_2)\wla_1\wla_2(\varepsilon^2+2\varepsilon/\gamma+1)\gamma}{(\wla_2 \varepsilon^2+\gamma(\wla_1+\wla_2)\varepsilon+\wla_1)^2}.
\end{equation}
Because of $\lambda(\varepsilon,\gamma)=\lambda(-\varepsilon,-\gamma)$
in \eqref{Ciaramella_mini_10_thm:prop}, we can assume without loss of generality that
$\gamma \geq 0$.

Let us now consider the case $\gamma=0$. In this case, the derivative \eqref{Ciaramella_mini_10_thm:der}
becomes $\frac{d \lambda(\varepsilon,0)}{d \varepsilon}
= \frac{(\lambda_1-\lambda_2)\wla_1\wla_2 2\varepsilon}{(\wla_2 \varepsilon^2+\wla_1^2)^2}$.
Moreover, since $\lambda(\varepsilon,0)=\lambda(-\varepsilon,0)$ we can assume that $\varepsilon \geq 0$.

Case ${\rm (B)}$.
If $\lambda_1>\lambda_2>0$, then $\frac{d \lambda(\varepsilon,0)}{d \varepsilon}>0$ for all $\varepsilon>0$.
Hence, $\varepsilon \mapsto \lambda(\varepsilon,0)$ is monotonically increasing, $\lambda(\varepsilon,0) \geq 0$ for all $\varepsilon>0$ and, thus, the minimum of $\varepsilon \mapsto |\lambda(\varepsilon,0)|$
is attained at $\varepsilon = 0$ with $|\lambda(0,0)|=|\lambda_2|>|\lambda_3|$, and the result follows.
Analogously, if $0>\lambda_2>\lambda_1$, then $\frac{d \lambda(\varepsilon,0)}{d \varepsilon}<0$ 
for all $\varepsilon>0$.
Hence, $\varepsilon \mapsto \lambda(\varepsilon,0)$ is monotonically decreasing, 
$\lambda(\varepsilon,0) < 0$ for all $\varepsilon>0$ and
the minimum of $\varepsilon \mapsto |\lambda(\varepsilon,0)|$
is attained at $\varepsilon = 0$.

Case ${\rm (C)}$.
If $\lambda_1>0>\lambda_2$, then $\frac{d \lambda(\varepsilon,0)}{d \varepsilon}>0$
for all $\varepsilon >0$. Hence, $\varepsilon \mapsto \lambda(\varepsilon,0)$ is monotonically 
increasing and such that $\lambda(0,0)=\lambda_2<0$ and $\lim_{\varepsilon \rightarrow \infty} \lambda(\varepsilon,0) = \lambda_1>0$. Thus, the continuity of the map $\varepsilon \mapsto \lambda(\varepsilon,0)$
guarantees the existence of an $\weps>0$ such that $\lambda(\weps,0)=0$.
Analogously, if $\lambda_2>0>\lambda_1$, then $\frac{d \lambda(\varepsilon,0)}{d \varepsilon}<0$
for all $\varepsilon>0$ and the result follows by the continuity of $\varepsilon \mapsto \lambda(\varepsilon,0)$.

Let us now consider the case $\gamma>0$. The sign of 
$\frac{d \lambda(\varepsilon,\gamma)}{d \varepsilon}$ is affected by the term
$f(\varepsilon):=\varepsilon^2+2\varepsilon/\gamma+1$, which appears at the numerator
of \eqref{Ciaramella_mini_10_thm:der}.
The function $f(\varepsilon)$ is strictly convex, attains its minimum
at $\varepsilon=-\frac{1}{\gamma}$, and is negative in $(\bareps_1,\bareps_2)$
and positive in $(-\infty,\bareps_1)\cup(\bareps_2,\infty)$, with $\bareps_1,\bareps_2=-\frac{1\mp \sqrt{1-\gamma^2}}{\gamma}$.

Case ${\rm (D)}$.
If $\lambda_1>\lambda_2>0$, then $\frac{d \lambda(\varepsilon,\gamma)}{d \varepsilon}>0$ for all 
$\varepsilon > \bareps_2$. Hence, $\frac{d \lambda(0,\gamma)}{d \varepsilon}>0$, which means that
there exists an $\weps<0$ such that $|\lambda(\weps,\gamma)|<|\lambda(0,\gamma)|=|\lambda_2|$.
The case $0>\lambda_2>\lambda_1$ follows analogously.

Case ${\rm (E)}$.
If $\lambda_1>0>\lambda_2$, then $\frac{d \lambda(\varepsilon,\gamma)}{d \varepsilon}>0$
for all $\varepsilon>0$. Hence, by the continuity of $\varepsilon \mapsto \lambda(\varepsilon,\gamma)$
(for $\varepsilon\geq 0$) there exists an $\weps>0$ such that $\lambda(\weps,\gamma)=0$.
The case $\lambda_2>0>\lambda_1$ follows analogously.
\end{proof}

Theorem \ref{Ciaramella_mini_10_thm:perturb} and its proof say that, if the two eigenvalues
$\lambda_1$ and $\lambda_2$ have opposite signs (but they could be equal in modulus),
then it is always possible to find an $\varepsilon \neq 0$ such that the 
coarse space $V_c := \Span\{ \bv_1 + \varepsilon \bv_2 \}$
leads to a faster method than $V_c := \Span\{ \bv_1 \}$, even though both are one-dimensional subspaces. 
In addition, if $\lambda_3 \neq 0$ the former leads to
a two-level operator $T$ with a larger kernel than the one corresponding to the latter.
The situation is completely different if
$\lambda_1$ and $\lambda_2$ have the same sign. In this case, the orthogonality parameter
$\gamma$ is crucial. If $\bv_1$ and $\bv_2$ are orthogonal ($\gamma=0$), then 
one cannot improve the effect of $V_c:= \Span\{ \bv_1 \}$ by a simple perturbation
using $\bv_2$. However, if $\bv_1$ and $\bv_2$ are not orthogonal ($\gamma \neq 0$),
then one can still find an $\varepsilon \neq 0$ such that $\rho(\wT(\varepsilon)) < \rho(\wT(0))$. 

Notice that, if $|\lambda_3|=|\lambda_2|$, Theorem \ref{Ciaramella_mini_10_thm:perturb} shows that one cannot obtain a $\rho(T)$ smaller than $|\lambda_2|$ using a one-dimensional perturbation. However, if one optimizes the entire coarse space $V_c$ (keeping $m$ fixed),
then one can find coarse spaces leading to better contraction factor of the two-level
iteration, even though $|\lambda_3|=|\lambda_2|$. This is shown in the next section.

\section{Optimizing the coarse-space functions}\label{Ciaramella_mini_10_sec:opt}
\vspace*{-4mm}

Consider the elliptic problem

\vspace*{-2mm}
\begin{equation}\label{Ciaramella_mini_10_eq:elliptic}
- \Delta u + c \, (\partial_x u + \partial_y u) = f \; \text{ in $\Omega=(0,1)^2$},\quad 
u = 0 \; \text{ on $\partial \Omega$}.
\end{equation}
\vspace*{-2mm}

\noindent
Using a uniform grid of size $h$,
the standard second-order finite-difference scheme for the Laplace operator
and the central difference approximation for the advection terms,
problem \eqref{Ciaramella_mini_10_eq:elliptic} becomes $A \bu = \bff$,
where $A$ has constant and positive diagonal entries, $D=\Diag(A)=4/h^2 I$.
A simple calculation
shows that, if $c\geq 0$ satisfies $c\leq 2/h$,
then the eigenvalues of $A$ are real. The eigenvectors of $A$ are
orthogonal if $c=0$ and non-orthogonal if $c>0$.

One of the most used smoothers for \eqref{Ciaramella_mini_10_eq:elliptic}
is the damped Jacobi method: $\bu^{k+1} = \bu^k + \omega D^{-1}( \bff - A \bu^k)$,
where $\omega \in (0,1]$ is a damping parameter.
The corresponding iteration matrix is $G=I-\omega D^{-1} A$.
Since $D=4/h^2 I$, the matrices $A$ and $G$ have the same eigenvectors.
For $c=0$, it is possible to show that, if $\omega=1$ (classical Jacobi iteration),
then the nonzero eigenvalues of $G$ have positive and negative signs,
while if $\omega=1/2$, the eigenvalues of $G$ are all positive.
Hence, the chosen model problem allows us to work in the theoretical framework 
of Section \ref{Ciaramella_mini_10_sec:pert}.

To validate numerically Theorem \ref{Ciaramella_mini_10_thm:perturb}, we set $h=1/10$ and consider $V_c:=\left\{\bv_1+\varepsilon \bv_2 \right\}$. Figure \ref{Ciaramella_mini_10_eq:validate_thm} shows 
the dependence of $\rho(T(\varepsilon))$ and $|\lambda(\varepsilon,\gamma)|$ on $\varepsilon$ and $\gamma$. On the top left panel, we set $c=0$ and $\omega=1/2$ so that the hypotheses of point (B) of Theorem \ref{Ciaramella_mini_10_thm:perturb} are satisfied, since $\gamma=0$ and $\lambda_1\geq \lambda_2>0$. As point (B) predicts, we observe that $\min\limits_{\varepsilon\in\mathbb{R}}\rho(T(\varepsilon))$ is attained at $\varepsilon=0$, i.e. $\min_{\varepsilon\in\mathbb{R}}\rho(T(\varepsilon))=\rho(T(0))=\lambda_2$. Hence, adding a perturbation does not improve the coarse space made only by $\bv_1$. Next, we consider point (C), by setting $c=0$ and $\omega=1$. Through a direct computation we get $\lambda_1=-0.95$, $\lambda_2=-\lambda_1$ and $\lambda_3=0.90$. The top-right panel shows, on the one hand, that for several values of $\varepsilon$, $\rho(T(\varepsilon))=\lambda_3<\lambda_2$, that is with a one-dimensional perturbed coarse space, we obtain the same contraction factor we would have with the two-dimensional spectral coarse space $V_c=\text{span}\left\{\bv_1,\bv_2\right\}$. On the other hand, we observe that there are two values of $\varepsilon$
such that $\rho(\widetilde{T}_{\widetilde{m}}(\varepsilon))=0$, which (recalling \eqref{Ciaramella_mini_10_eq:wT} and \eqref{Ciaramella_mini_10_eq:wTT})
implies that $T$ is nilpotent over the $\Span\{\bv_1,\bv_2\}$.
To study point (D), we set $c=10$, $\omega=1/2$, which lead to $\lambda_1=0.92$, $\lambda_2=\lambda_3=0.90$. The left-bottom panel confirms there exists an $\varepsilon^*<0$ such that $|\lambda(\varepsilon^*,\gamma)|\leq \lambda_2$, which implies $\rho(T(\varepsilon^*))\leq \lambda_2$.
Finally, we set $c=10$ and $\omega=1$. Point (E) is confirmed by the right-bottom panel, which shows that $|\lambda(\varepsilon,\gamma)|<|\lambda_2|$, and thus $\min_{\varepsilon}\rho(T(\varepsilon))=|\lambda_3|$, for some values of $\varepsilon$.

\begin{figure}[t]
\centering
\includegraphics[scale=0.32]{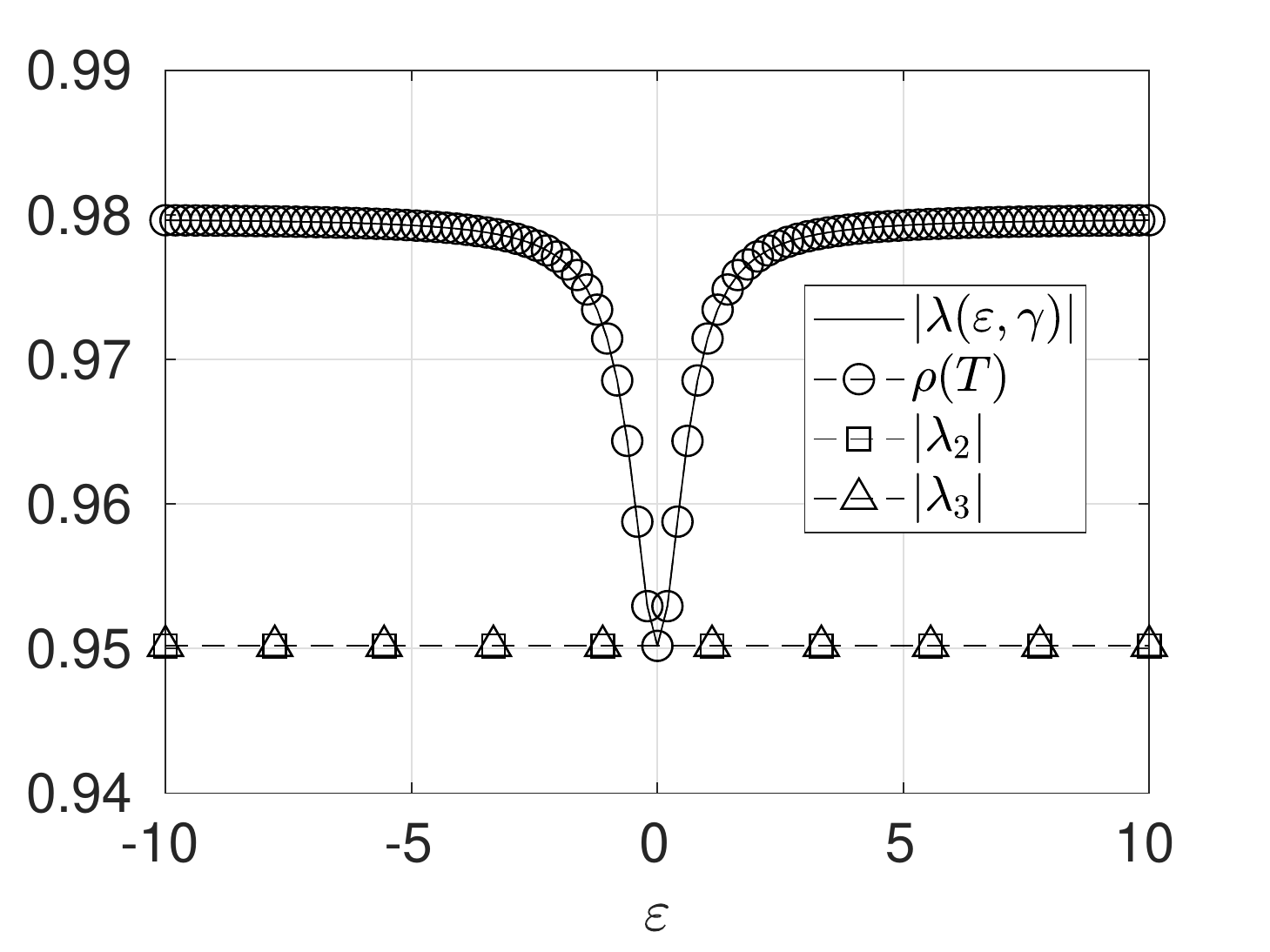}
\includegraphics[scale=0.32]{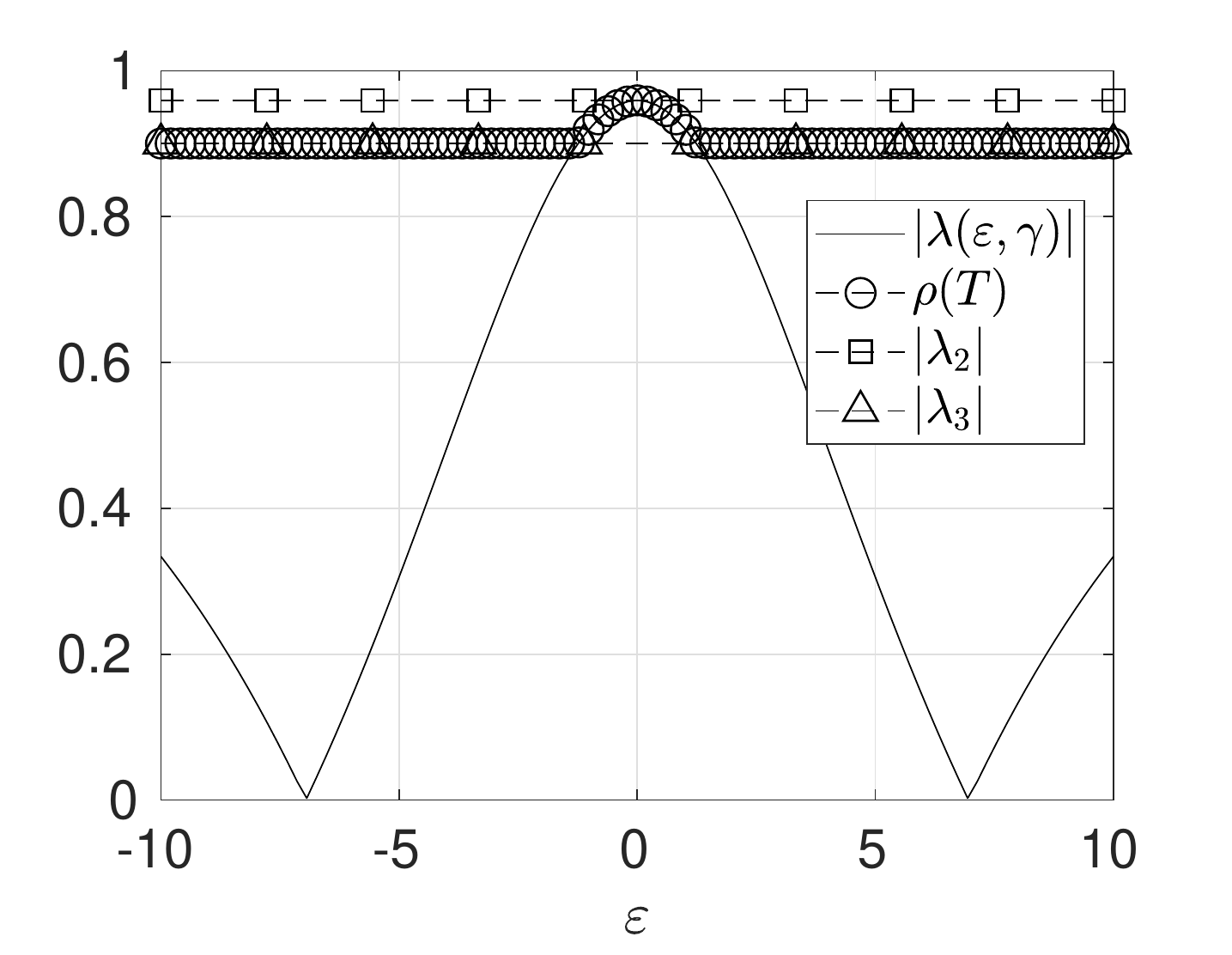}
\includegraphics[scale=0.32]{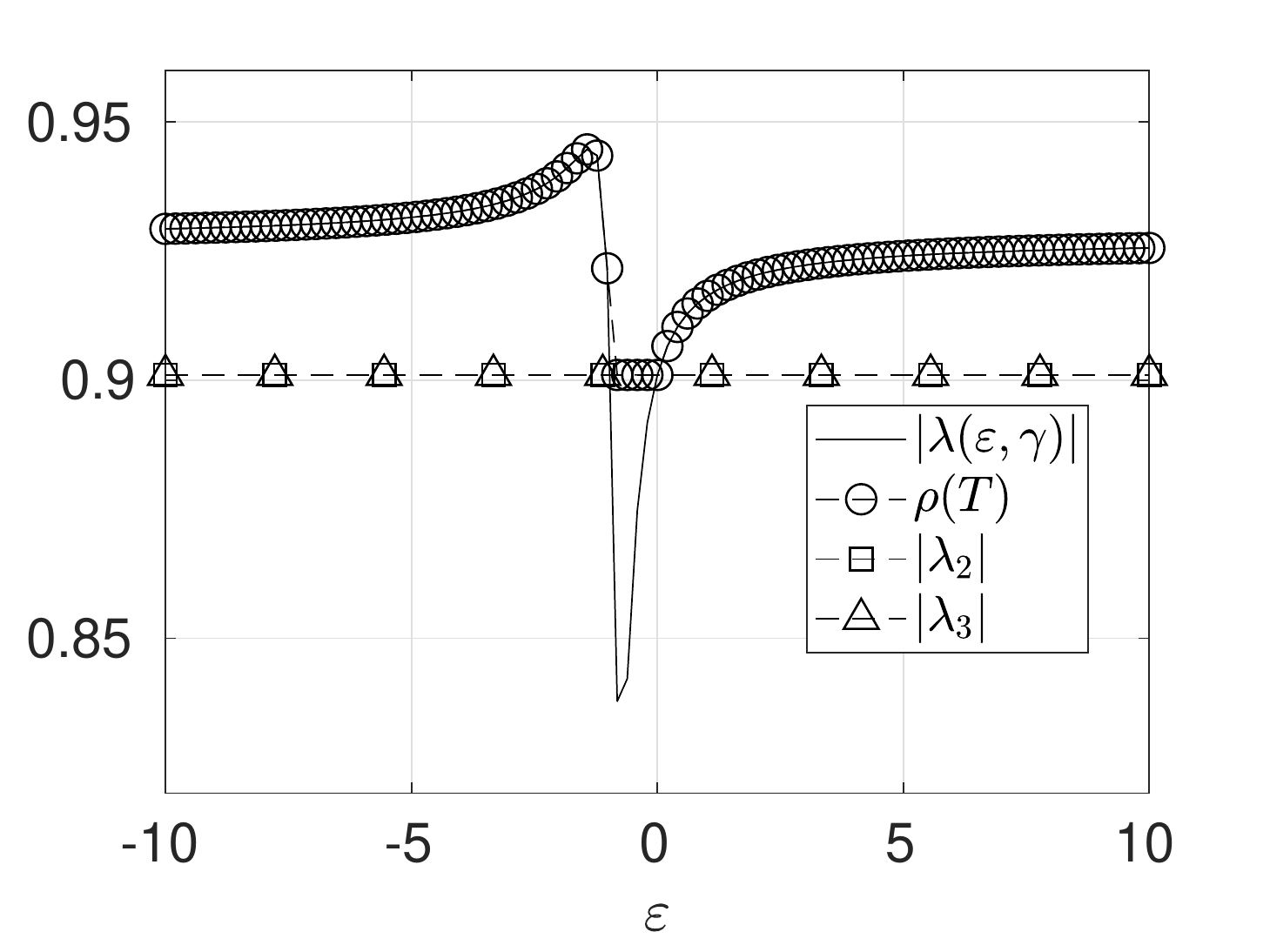}
\includegraphics[scale=0.32]{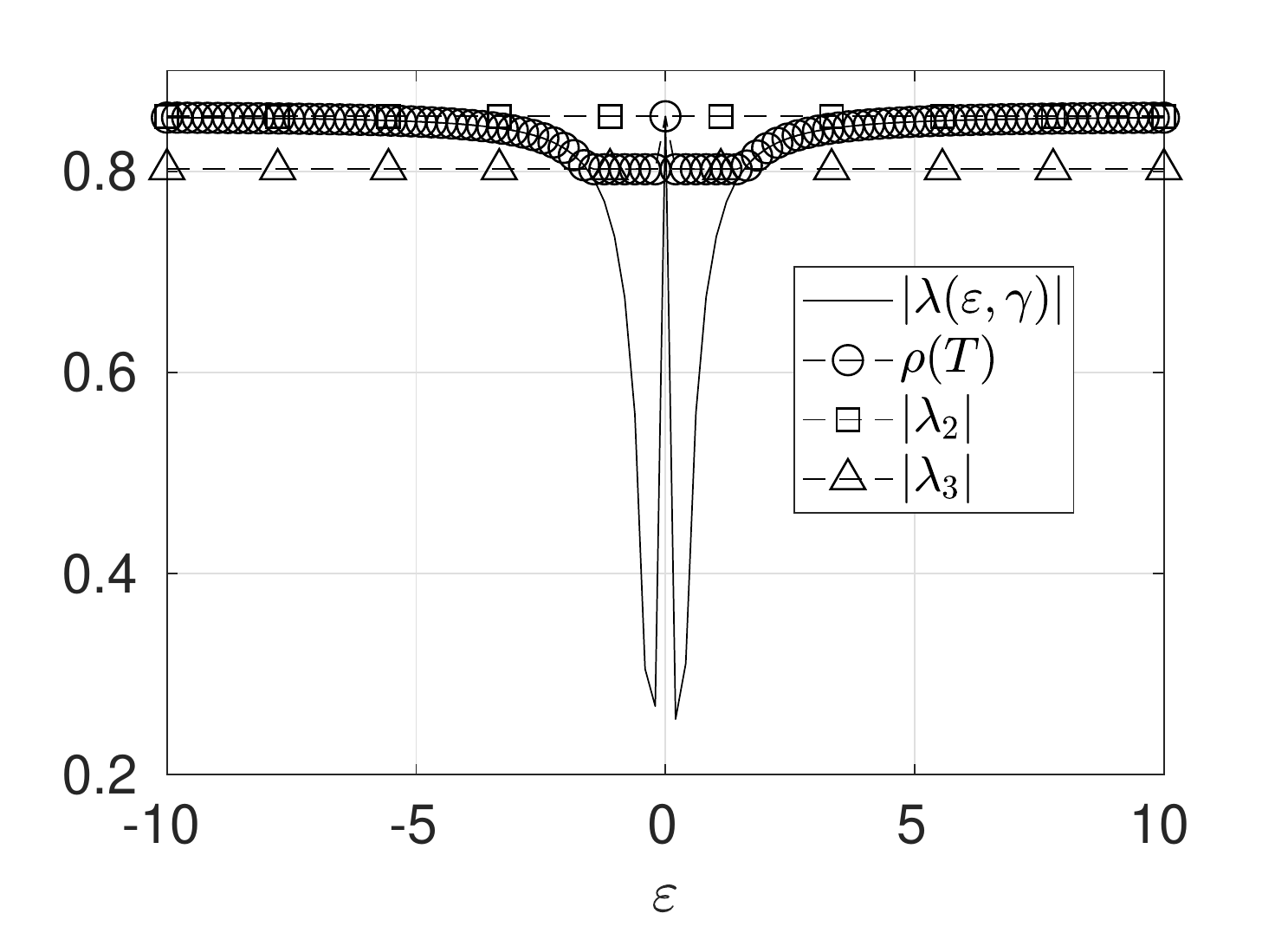}
\caption{Behavior of $|\lambda(\varepsilon,\gamma)|$ and $\rho(T(\varepsilon))$ as functions of $\varepsilon$ for different $c$ and $\gamma$.}\label{Ciaramella_mini_10_eq:validate_thm}
\end{figure}

We have shown both theoretically and numerically that the spectral coarse space
is not necessarily the one-dimensional coarse space minimizing $\rho(T)$. 
Now, we wish to go beyond this one-dimensional analysis and optimize
the entire coarse space $V_c$ keeping its dimension $m$ fixed). This is equivalent to
optimize the prolongation operator $P$ whose columns span $V_c$. 
Thus, we consider the optimization problem
\begin{equation}\label{Ciaramella_mini_10_eq:optimization_problem}
\min_{P \in\mathbb{R}^{n\times m}} \rho(T(P)).
\end{equation}
To solve approximately \eqref{Ciaramella_mini_10_eq:optimization_problem}, we follow 
the approach proposed by \cite{katrutsa2017deep}.
Due to the Gelfand formula $\rho(T)=\lim_{k\rightarrow \infty} \sqrt[k]{\|T^k\|_F}$, we replace \eqref{Ciaramella_mini_10_eq:optimization_problem} with the simpler optimization problem $\min_{P} \|T(P)^k\|^2_F$ for some positive $k$. Here, $\|\cdot\|_F$ is the Frobenius norm. We then consider the unbiased stochastic estimator \cite{hutchinson1989stochastic}
\[\|T^k\|^2_F=\text{trace}\left((T^k)^\top T^k\right)=\mathbb{E}_{\mathbf{z}}\left[ \mathbf{z}^\top (T^k)^\top T^k  \mathbf{z}\right]=\mathbb{E}_{\mathbf{z}}\left[ \|T^k \mathbf{z}\|^2_2\right] ,\]
where $\mathbf{z}\in\mathbb{R}^n$ is a random vector with Rademacher distribution, i.e. $\mathbb{P}(\mathbf{z}_i=\pm 1)=1/2$. Finally, we rely on a sample average approach, replacing the unbiased stochastic estimator with its empirical mean such that \eqref{Ciaramella_mini_10_eq:optimization_problem} is approximated by
\begin{equation}\label{Ciaramella_mini_10_eq:optimization_problem_emp}
\min_{P \in\mathbb{R}^{n\times m}} \frac{1}{N}\sum_{i=1}^N\|T(P)^k \mathbf{z}_i\|^2_F, 
\end{equation}
where $\mathbf{z}_i$ are a set of independent, Rademacher distributed, random vectors.
The action of $T$ onto the vectors $\mathbf{z}_i$ can be interpreted as the feed-forward process of a neural net, where each layer represents one specific step of the two-level method, that is the smoothing step, the residual computation, the coarse correction and the prolongation/restriction operations. In our setting, the weights of most layers are fixed and given, and the optimization is performed only on the weights of the layer representing the prolongation step. The restriction layer is constraint to have as weights the transpose of the weights of the prolongation layer.

We solve \eqref{Ciaramella_mini_10_eq:optimization_problem_emp}
for $k=10$ and $N=n$ using Tensorflow
\cite{tensorflow2015-whitepaper} and its stochastic gradient descend algorithm 
with learning parameter 0.1.
The weights of the prolongation layer are initialized with an uniform distribution. 
Table \ref{Ciaramella_mini_10_eq:Tab} reports both $\rho(T(P))$ and $\|T(P)\|_A$ using a spectral 
coarse space and the coarse space obtained solving \eqref{Ciaramella_mini_10_eq:optimization_problem_emp}.
\begin{table}[t]
\centering
\def\arraystretch{1.2}
\setlength{\tabcolsep}{5pt}
\begin{tabular}{ l |c c | c c c  c}
& $c$ & $\omega$ & $m=1$ & $m =5$ & $m=10$ & $m=15$\\ 
\hline
\parbox[b][-24pt][c]{8pt}{\rotatebox[origin=c]{90}{$\rho(T)$}} & 0 & 1/2 & 0.95 - 0.95 & 0.90 - 0.90 & 0.82 - 0.83 & 0.76 - 0.78 \\
& 0 & 1 & 0.95 - 0.90 & 0.90 - 0.80 & 0.80 - 0.65 & 0.74 - 0.53 \\
& 10 & 1/2 & 0.90 - 0.90  & 0.85 - 0.82  & 0.79 - 0.74 & 0.73 - 0.68 \\
& 10 & 1 & 0.85 - 0.80 & 0.80 - 0.67 & 0.71 - 0.55 & 0.66 - 0.37 \\ \hline
\parbox[b][-7pt][c]{8pt}{\rotatebox[origin=c]{90}{$\|T\|_{A}$}} 
& 0 & 1/2 & 0.95 - 0.95 & 0.90 - 0.90 & 0.82 - 0.84 & 0.76 - 0.77 \\
& 0 & 1 & 0.95 - 0.95 & 0.90 - 0.94 & 0.80 - 0.88 & 0.74 - 0.88 \\ \hline
\parbox[b][-8pt][c]{8pt}{\rotatebox[origin=c]{90}{$\kappa_2$}} 
& 0 & 1 & 46.91 - 29.45 & 18.48 - 14.40 & 9.37 - 8.22 & 6.69 - 8.53 \\
& 10 & 1 & 27.25 - 23.98 & 22.44 - 12.36 & 17.34 - 11.35 & 13.06 - 9.71 \\
\end{tabular}
\caption{Values of $\rho(T)$, $\|T\|_{A}$ and condition number $\kappa_2$ of
the matrix $A$ preconditioned by the two-level method for different
 $c$ and $\omega$ and using either a spectral coarse space (left number), or the coarse space obtained solving \eqref{Ciaramella_mini_10_eq:optimization_problem_emp} (right number).}\label{Ciaramella_mini_10_eq:Tab}
\end{table}
We can clearly see that there exist coarse spaces, hence matrices $P$, 
corresponding to values of the asymptotic convergence factor $\rho(T(P))$
much smaller than the ones obtained by spectral coarse spaces. 
Hence, Table \ref{Ciaramella_mini_10_eq:Tab} confirms that a spectral coarse space of dimension $m$ 
is not necessarily a (global) minimizer for $\min\limits_{P \in \mathbb{R}^{n\times m}}\rho(T(P))$.
This can be observed not only in the case $c=0$, for which the result of \cite[Theorem 5.5]{xu_zikatanov_2017} 
states that (recall that $M$ is symmetric) the spectral coarse space minimizes $\|T(P)\|_A$,
but also for $c> 0$, which corresponds to a nonsymmetric $A$.
Interestingly, the coarse spaces obtained by our numerical optimizations
lead to preconditioned matrices with better condition numbers,
as shown in the last row of Table \ref{Ciaramella_mini_10_eq:Tab},
where the condition number $\kappa_2$ of the matrix $A$ preconditioned by the two-level method
(and different coarse spaces) is reported.



	
	\bibliographystyle{plain}
	\bibliography{Ciaramella_mini_10}
\end{document}